\documentclass[a4paper,10pt]{article}
\usepackage[latin1]{inputenc}
\usepackage[T1]{fontenc}
\usepackage{amsmath}
\usepackage{amsfonts}
\usepackage{amstext}
\usepackage{amsthm}
\usepackage[all]{xy}
\usepackage{geometry}
\usepackage{srcltx}
\usepackage{graphics}
\usepackage{epsfig}
\usepackage{subfigure}
\usepackage{slashed}
\usepackage{color}
\usepackage{amssymb}
\usepackage{srcltx}

\newcommand{\ip}[2]{\langle #1, #2 \rangle}
\newcommand{\abs}[1]{|#1|}
\newcommand{\norm}[1]{||#1||}

\newcommand{\ls}{\mathcal{LS}}
\newcommand{\bds}{(BD) }
\newcommand{\bd}{(BD)}
\newcommand{\wbds}{(\rm{wBD}) }
\newcommand{\wbd}{(\rm{wBD})}

\newtheorem{theorem}{Theorem}[section]
\newtheorem*{theorem*}{Arnold Conjecture}
\newtheorem{lemma}[theorem]{Lemma}
\newtheorem{defi}[theorem]{Definition}

\newtheorem{prop}[theorem]{Proposition}
\newtheorem{cor}[theorem]{Corollary}

\DeclareMathOperator{\ch}{ch}

\DeclareMathOperator{\inv}{Inv}
\DeclareMathOperator{\interior}{int}

\DeclareMathOperator{\crit}{Crit}

\DeclareMathOperator{\CL}{CL}

\title{The $E$-cohomological Conley Index, Cup-Lengths and the Arnold Conjecture on $T^{2n}$}
\author{Maciej Starostka and Nils Waterstraat}

\begin{document}
\date{}
\maketitle


\begin{abstract}
\noindent  
We give a new proof of the strong Arnold conjecture for $1$-periodic solutions of Hamiltonian systems on tori, that was first shown by C. Conley and E. Zehnder in 1983. Our proof uses other methods and is shorter than the previous one. We first show that the $E$-cohomological Conley index, that was introduced by the first author recently, has a natural module structure. This yields a new cup-length and a lower bound for the number of critical points of functionals. Then an existence result for the $E$-cohomological Conley index, which applies to the setting of the Arnold conjecture, paves the way to a new proof of it on tori.    
\end{abstract}

\section{Introduction}
Motivated by questions of celestial mechanics from the beginning of the 20th century, Arnold conjectured in the sixties that every Hamiltonian diffeomorphism on a compact symplectic manifold $(M,\omega)$ has at least as many fixed points as a function on $M$ has critical points. Let us recall that a diffeomorphism $\psi:M\rightarrow M$ is called Hamiltonian if there exists a smooth map $H:\mathbb{R}\times M\rightarrow\mathbb{R}$, $H(t+1,x)=H(t,x)$, such that $\psi=\eta^1$, where the family $\{\eta^t\}_{t\in\mathbb{R}}$ satisfies

\begin{equation}\label{flow}
\left\{
\begin{aligned}
\frac{d}{dt}\eta^t&=X_H(\eta^t)\\
\eta^0&=id,
\end{aligned}
\right.
\end{equation}
and $X_H$ stands for the time-dependent vector field given by

\[dH(\cdot)=\omega(X_H,\cdot).\]   
Consequently, $p$ is a fixed point of $\psi$ if and only if it is the initial condition of a $1$-periodic solution of \eqref{flow}, and so Arnold's famous conjecture can be reformulated dynamically as follows.

\begin{theorem*}
The Hamiltonian system

\begin{align}\label{Arnold}
\dot{x}(t)=X_H(x(t))
\end{align}
has at least as many $1$-periodic orbits as a function on $M$ has critical points.
\end{theorem*}
\noindent
The aim of this paper is to point out a new approach to the Arnold conjecture which proves it on tori, where it was first shown by C. Conley and A. Zehnder in \cite{ConleyZehnder}. It will be future work to investigate if our methods also apply to cases where the conjecture is still open. However, let us point out that, apart from these important applications, our methods are of independent interest and can be outlined as follows.\\
In \cite{Maciej} the first author introduced the $E$-cohomological Conley index for isolated invariant sets of flows in Hilbert spaces. Roughly speaking, it is a generalisation of the classical Conley index for flows on locally compact spaces by using $E$-cohomology, which is a generalised cohomology theory for subsets of Hilbert spaces that was constructed by Abbondandolo in \cite{Alberto} (cf. also \cite{Geba}). The first aim of this paper is to introduce a module structure for the $E$-cohomological Conley index, which allows us to define a relative cup-length for triples of closed and bounded subsets of Hilbert spaces. Secondly, we consider this numerical invariant for isolating neighbourhoods of $\ls$-flows in Hilbert spaces (cf. \cite{Marek}, \cite{Marcin}), and show that it is a lower bound for the number of critical points of gradient flows as in classical Ljusternik-Schnirelman theory. Here we substantially use the homotopy invariance of the $E$-cohomological Conley index that was recently obtained by the first author in a joint work with Izydorek, Rot, Styborski and Vandervorst in \cite{Invariance}. Thirdly, we introduce a Conley index for unbounded isolating neighbourhoods of $\ls$-flows, and prove a sufficient condition for its existence. To the best of our knowledge, no such result has been obtained before in the literature. Finally, and most important, we show that this sufficient condition is satisfied when dealing with the functionals in the setting of the Arnold conjecture on $T^{2n}$. This yields an estimate from below for the number of contractible $1$-periodic solutions of \eqref{Arnold} by one of our previous results, and the obtained bound is indeed the one that Arnold conjectured. Let us point out that our proof of the Arnold conjecture not only differs substantially from Conley and Zehnder's, but it is also much shorter. A key step in our argument is to show that we can use the homotopy invariance of the $E$-cohomological Conley index from \cite{Invariance} to deform the Hamiltonian to a constant function which simplifies the computations considerably.\\
This paper is organised with the intention of guiding the reader through our proof of the Arnold conjecture in as straightforward a manner as possible. Therefore, in the second section, we only introduce the material that is necessary to understand the basics of our approach and postpone more technical proofs to Section 4. Our discussion of the Arnold conjecture can be found in between, in the third section.

\subsubsection*{Acknowledgements}
We would like to thank Kazimierz G\c{e}ba and Marek Izydorek for many inspiring discussions, as well as Thomas Schick for clarifying remarks about our groups $H^\ast_0(X)$.


\section{The $E$-Cohomological Conley Index and Cup-Lengths}

\subsection{Module Structure for $E$-Cohomology}

We begin this section by recalling the definition of $E$-cohomology from \cite{Alberto}, and to this aim we let $E$ be a separable real Hilbert space and $E^+$, $E^-$ closed subspaces such that $E=E^+\oplus E^-$. We endow $E^+$ with the weak topology, $E^-$ with the strong topology and henceforth we consider $E$ with the corresponding product topology.\\ 
In what follows we denote by $H^\ast$ Alexander-Spanier cohomology with compact supports, for which we refer to \cite{Spanier} and the nice survey in \cite[\S 1]{Alberto}. Moreover, we let $\mathcal{V}$ be the set of all finite dimensional subspaces of $E^-$, which is partially ordered by inclusion and directed.\\
If $U,V,W\in\mathcal{V}$ are such that $W=V\oplus U$ and $\dim(U)=1$, then we can decompose $W$ into two subspaces by setting

\begin{align*}
W^+&=\{w\in W:\, \langle w,u\rangle\geq 0\}\\
W^-&=\{w\in W:\, \langle w,u\rangle\leq 0\},
\end{align*}
where $u\neq 0$ is a fixed element in $U$. Note that the choice of $u$ corresponds to an orientation of the one-dimensional space $U$, and changing this orientation swaps $W^+$ and $W^-$.\\
We set for a closed and bounded subset $X$ of $E$

\[X_W=X\cap(E^+\times W),\quad X^+_W=X\cap(E^+\times W^+),\quad X^-_W=X\cap(E^+\times W^-)\]   
and note that $X_W=X^+_W\cup X^-_W$ as well as $X_V:=X\cap(E^+\times V)=X^+_W\cap X^-_W$. 
\begin{center}\includegraphics[height=7cm]{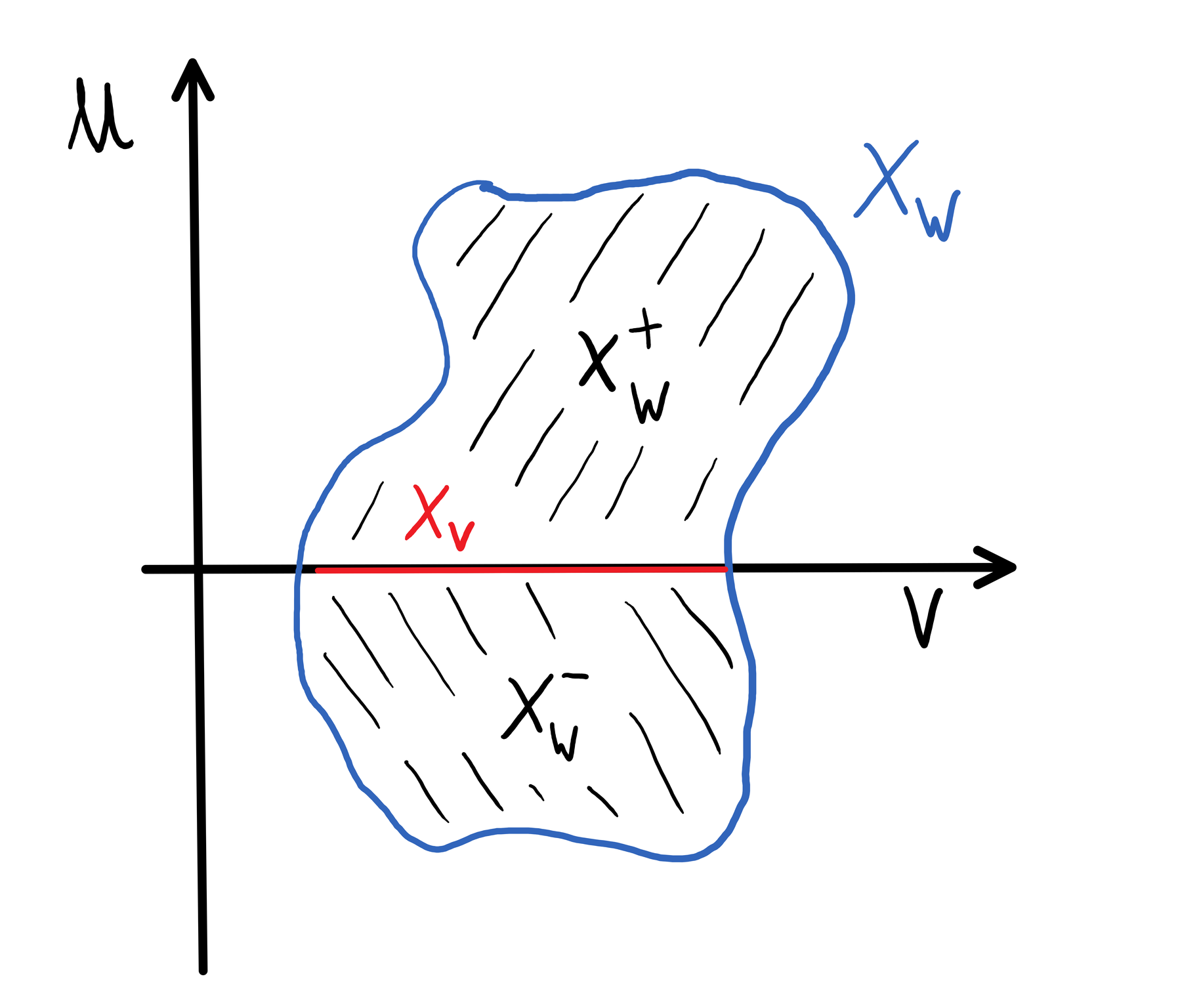}\end{center}
If now $A\subset X$ is closed, then we obtain a relative Meyer-Vietoris sequence

\begin{align*}
\ldots&\rightarrow H^k(X^+_W,A^+_W)\oplus H^k(X^-_{W},A^-_W)\rightarrow H^k(X_{V},A_V)\xrightarrow{\Delta^k_{V,W}}\\ 
&\rightarrow H^{k+1}(X_{W},A_W)\rightarrow H^{k+1}(X^+_W,A^+_W)\oplus H^{k+1}(X^-_W,A^-_W)\rightarrow\ldots.
\end{align*} 
In the more general case that $W=V\oplus U$ and $\dim(U)=n>0$, we decompose $U$ into $n$ one-dimensional subspaces $U=U_1\oplus\dots\oplus U_n$ and set $W_i=V\oplus U_1\oplus\dots\oplus U_i$ for $1\leq i\leq n$ as well as $W_0=V$. Then the previous construction yields $n$ Mayer-Vietoris homomorphisms

\[\Delta^{k+i-1}_{W_{i-1},W_i}:H^{k+i-1}(X_{W_{i-1}},A_{W_{i-1}})\rightarrow H^{k+i}(X_{W_{i}},A_{W_{i}})\]
and their composition is a homomorphism $H^k(X_V,A_V)\rightarrow H^{k+n}(X_W,A_W)$. Hence we have constructed for any $q\in\mathbb{Z}$ and $V,W\in\mathcal{V}$, $V\subset W$, a homomorphism

\[\Delta^q_{V,W}(X):H^{q+\dim(V)}(X_V,A_V)\rightarrow H^{q+\dim(W)}(X_W,A_W).\]
As noted in \cite[Prop. 2.2]{Alberto}, these maps do not depend on the choice of the one-dimensional subspaces $U_i$ and their orientations. 
In summary, $\{H^{q+\dim(V)}(X_V,A_V),\Delta^q_{VW}(X,A)\}$ is a direct system of abelian groups over the directed set $\mathcal{V}$.

\begin{defi}
Let $A\subset X$ be closed and bounded subsets of $E$. The $E$-cohomology group of index $q\in\mathbb{Z}$ of $(X,A)$ is the direct limit

\[H^q_E(X,A)=\lim_{\stackrel{\longrightarrow}{V\in\mathcal{V}}}\{H^{q+\dim(V)}(X_V,A_V),\Delta^q_{V,W}(X,A)\},\]
and we set as usual $H^q_E(X):=H^q_E(X,\emptyset)$. 
\end{defi}
\noindent
The inclusions $\iota_{V,W}:X_V\rightarrow X_{W}$ for $V,W\in\mathcal{V}$ yield an inverse system $\{H^p(X_V),\iota^\ast_{V,W}\}$ over $\mathcal{V}$. We define for $p\in\mathbb{Z}$ the group $H^p_0(X)$ as the inverse limit
\[H^p_0(X):=\lim_{\stackrel{\longleftarrow}{V\in\mathcal{V}}}\{H^p(X_V),\iota^\ast_{V,W}\}.\]
In what follows, we denote elements of $H^p_0(X)$ by $[\alpha_V]_0$ if $\alpha_V\in H^p(X_V)$, and correspondingly elements of $H^q_E(X,A)$ by $[\alpha_V]_E$ if $\alpha_V\in H^{q+\dim(V)}(X_V,A_V)$.\\
Let us point out that $H^\ast_0(X)$ is a ring if we define the product of $[\alpha_V]_0\in H^p_0(X)$ and $[\beta_V]_0\in H^q_0(X)$ by

\[[\alpha_V]_0\cup [\beta_V]_0=[\alpha_{V}\cup\beta_{V}]_0\in H^{p+q}_0(X).\]
It is readily seen from the naturality of the cup product that this is a sensible definition. 


\begin{prop}\label{prop-module}
$H^\ast_E(X,A)$ is a right module over $H^\ast_0(X)$, where the module multiplication is induced by the cup product. 
\end{prop}

\begin{proof}
We define for $[\alpha_V]_0\in H^r_0(X)$ and $[\beta_V]_E\in H^q_E(X,A)$

\[[\beta_V]_E\cup[\alpha_V]_0:=[\beta_V\cup\alpha_V]_E\in H^{q+r}_E(X,A).\]
This product is well defined, as if $\beta_W=\Delta^q_{V,W}\beta_V$ and $\alpha_V=\iota^\ast_{V,W}\alpha_W$, then

\[\Delta^{q+r}_{V,W}(\beta_V\cup\alpha_V)=\Delta^{q+r}_{V,W}(\beta_V\cup\iota^\ast_{V,W}\alpha_W)=(\Delta^q_{V,W}\beta_V)\cup\alpha_W=\beta_W\cup\alpha_W,\]
where we have used that the coboundary operators of the Mayer-Vietoris sequence commute with products in multiplicative cohomology theories (cf. \cite[Prop. 17.2.1]{Tammo}). 
\end{proof}
\noindent
Let now $\Omega\subset E$ be closed and bounded and such that $X\subset\Omega$. The inclusions $j_V:X_V\rightarrow\Omega_V$ induce homomorphisms $j^\ast_V:H^p(\Omega_V)\rightarrow H^p(X_V)$ for $V\in\mathcal{V}$, and it is readily seen that they actually yield a ring homomorphism 

\[j^\ast:H^\ast_0(\Omega)\rightarrow H^\ast_0(X).\]
Consequently, we obtain the following corollary from Proposition \ref{prop-module}.





\begin{cor}
For every $X\subset\Omega\subset E$, $H^\ast_E(X,A)$ is a right $H^\ast_0(\Omega)$-module.
\end{cor}
\noindent
Henceforth we denote the module product of $\alpha\in H^r_0(\Omega)$ and $\beta\in H^p_E(X,A)$ by

\[\beta\cup\alpha\in H^{p+r}_E(X,A).\]
We conclude this section with the following crucial definition.

\begin{defi}\label{cuplength}
Let $A\subset X\subset\Omega$ be closed and bounded subsets of $E$. The \textbf{relative cup-length} $\CL(\Omega;X,A)$ is defined as follows:

\begin{itemize}
	\item If $H^\ast_E(X,A)=0$, we set
	
	\[\CL(\Omega;X,A)=0.\]
	\item If $H^\ast_E(X,A)\neq 0$ but $\beta\cup\alpha=0$ for every $\beta\in H^\ast_E(X,A)$ and $\alpha\in H^{>0}_0(\Omega)$, then we set
	
	\[\CL(\Omega;X,A)=1.\]
	
	\item If there are $k\geq 2$, $\beta_0\in H^\ast_E(X,A)$ and $\alpha_1, \alpha_2, \ldots, \alpha_{k-1} \in H^{>0}_0(\Omega)$ 
	such that 
	
	\[\beta_0\cup\alpha_1 \cup \ldots \cup \alpha_{k-1}\neq 0\] 
	but 
	
	\[\beta\cup \gamma_1 \cup \ldots \cup\gamma_{k}=0\]
	 for all $\beta\in H^\ast_E(X,A)$, $\gamma_1, \gamma_2, \ldots, \gamma_{k}
	\in H^{>0}_0(\Omega)$, then
	
	\[\CL(\Omega;X,A)=k.\] 
\end{itemize}
\end{defi}


\subsection{The $E$-Cohomological Conley Index and Critical Points}
The first aim of this section is to introduce the $E$-cohomological Conley index and to define a module structure for it. Let $E$ be a real separable Hilbert space and $L:E\rightarrow E$ an invertible selfadjoint operator for which there exists a sequence $\{E_n\}_{n\in\mathbb{N}}$ of finite dimensional subspaces of $E$ such that $L(E_n)=E_n$, $E_n\subset E_{n+1}$ and $\overline{\bigcup_{n\in\mathbb{N}}E_n}=E$. Let $U\subset E$ be open. Following \cite{Marek}, we call a vector field $F:U\subset E\rightarrow E$, $F(u)=Lu+K(u)$ an $\ls$-vector field if $K:U\subset E\rightarrow E$ is a locally Lipschitz compact operator. Note that every $\ls$-vector field generates a local flow $\eta^t$ satisfying

\[\frac{d}{dt}\eta^t=-F\circ\eta^t,\quad \eta^0=id,\]
which we call an $\ls$-flow. \\
Let us now assume that $\eta$ is an $\ls$-flow on $U$, and let us denote by

\[\inv(\Omega,\eta)=\{x\in\Omega:\,\eta^t(x)\in\Omega,\, t\in\mathbb{R}\}\]
the maximal $\eta$-invariant subset of $\Omega\subset U$.

\begin{defi}\label{isolnbhd}
A closed and bounded set $\Omega\subset U$ is called an \textit{isolating neighbourhood} of $\eta$ if $\inv(\Omega,\eta)\subset\interior(\Omega)$, where $\interior(\Omega)$ denotes the interior of $\Omega$. 
\end{defi}
\noindent
Let now $\Omega$ be an isolating neighbourhood of $\eta$ and $S:=\inv(\Omega,\eta)$.

\begin{defi}
We call a closed and bounded pair $(X,A)$ of subsets of $\Omega$ an \textit{index pair} for $S$ if

\begin{itemize}
	\item $A$ is positively invariant with respect to $X$, i.e. given $x\in A$ and $t>0$ with $\eta^{[0,t]}(x)\subset X$, then $\eta^{[0,t]}(x)\in A$.
	\item $S = \inv (\overline{X \setminus A},\eta) \subset \interior(\overline{X \setminus A})$,
	\item if $y\in X$, $t>0$ and $\eta(t,y)\notin X$, then there exists $t'>t$ such that $\eta^{[0,t']}(y)\subset X$ and $\eta(t',y)\in A$.
\end{itemize} 
\end{defi}
\noindent
It was shown in \cite[Lemma 2.7]{Invariance} that every isolated invariant set $S$ has an index pair.\\
Note that the space $E$ splits as $E=E^+\oplus E^-$, where $E^\pm$ are the spectral subspaces with respect to the positive and negative part of the spectrum of $L$. Henceforth, we let $H^\ast_E(X,A)$ be the E-cohomology groups with respect to this decomposition. The following crucial result was proved in \cite[Prop. 2.8]{Invariance}.

\begin{prop}\label{moduleiso}
If $\Omega$ is an isolating neighbourhood of $\eta$, $S=\inv(\Omega,\eta)$ and $(X,A)$, $(X',A')$ are index pairs for $S$ such that $X,X'\subset\Omega$, then the groups $H^\ast_E(X,A)$ and $H^\ast_E(X',A')$ are isomorphic.
\end{prop} 
\noindent
Hence the next definition is sensible (cf. \cite[Def. 2.9]{Invariance}).

\begin{defi}
The \textit{$E$-cohomological Conley index} of $S$ is defined by

\[\ch_E(S)=H^\ast_E(X,A),\]
where $(X,A)$ is an index pair for $S$.
\end{defi}
\noindent
If we want to emphasize the isolating neighbourhood $\Omega$ instead of the isolated invariant set $S$, we will also write $\ch_E(\Omega)$ to denote the $E$-cohomological Conley index.\\
When taking the module structure from \S 2 into account, it is readily seen by arguing as in \cite[Prop. 2.8]{Invariance} that $H^\ast_E(X,A)$ and $H^\ast_E(X',A')$ are actually isomorphic as $H^\ast_0(\Omega)$-modules. Hence we obtain as a consequence of Proposition \ref{moduleiso} the following important result.

\begin{cor}\label{cor-moduleiso}
The cup-length $\CL(\Omega;X,A)$ does not depend on the choice of the index pair $(X,A)$ such that $X\subset\Omega$.
\end{cor}
\noindent
Consequently, we can define

\[\CL(\Omega,S):=\CL(\Omega;X,A),\]
where $(X,A)$ is any index pair such that $X\subset\Omega$. As $S$ is uniquely determined by $\Omega$ and the flow $\eta$, we will sometimes denote this cup-length by $\CL(\Omega,\eta)$ if we want to emphasize $\eta$.\\
Before we come to the main theorem of this section, we want to state the following important result for later reference. It follows from the corresponding Theorem 2.12 in \cite{Invariance} when taking the module structure of $H^\ast_E$ into account.

\begin{theorem}\label{continuation}
If $\{\eta_\lambda, \lambda \in [0,1]\}$ is a continuous family of $\ls$-flows on $U$ such that $\Omega\subset U$ is an isolating neighbourhood of $\eta_\lambda$ for every $\lambda \in [0,1]$, then 

\[\CL(\Omega,\eta_0)=\CL(\Omega,\eta_1).\]
\end{theorem}
\noindent
Let us now assume that $\eta$ is the gradient flow with respect to a differentiable functional $f:U\rightarrow\mathbb{R}$, i.e. the map $F:U\subset E\rightarrow E$ is of the form $F=\nabla f$. Let $\Omega$ be an isolating neighbourhood of $\eta$ and $S=\inv(\Omega,\eta)$. We denote by $\crit(f,\Omega)$ the set of critical values of $f\mid_{\Omega}$ and can now state the first important result of this paper.

\begin{theorem}\label{thm-estimate} If $f$ has only finitely many critical points in $\Omega$, then the number of critical values of $f\mid_{\Omega}$ is bounded below by the cup-length of $\Omega$ with respect to $S$, i.e.
\begin{align}\label{CLestimate}
\#\crit(f,\Omega)\geq\CL(\Omega,S).
\end{align}
\end{theorem}
\noindent
Note that by Theorem \ref{thm-estimate}, the right hand side in \eqref{CLestimate} is obviously also a lower bound for the number of critical points of $f$ in $\Omega$. We will prove Theorem \ref{thm-estimate} in Section \ref{section-proof}.


\subsection{Unbounded isolating neighbourhoods}
Let us recall that we required an isolating neighbourhood in Definition \ref{isolnbhd} to be closed and bounded. For our proof of the Arnold conjecture, we need to broaden our point of view and now consider also unbounded isolating neighbourhoods. As before, we assume that $\eta$ is an $\ls$-flow on the open subset $U$ of the separable real Hilbert space $E$.

\begin{defi}\label{isolnbhdunb}
A closed and unbounded set $\Omega\subset U$ is called an \textit{unbounded isolating neighbourhood} of $\eta$ if $\inv(\Omega,\eta)\subset\interior(\Omega)$. 
\end{defi}
\noindent
Note that, in general, the Conley index is not well defined for an unbounded isolating neighbourhood $\Omega$. However, if the invariant set $S = \inv(\Omega,\eta)$ is bounded, then it is contained in a ball of radius $R'$ and $\Omega^R:=\Omega \cap B(R)$ is an isolating neighbourhood of $\eta$ for every $R \geq R'$. It follows from Proposition \ref{moduleiso} that $\ch_E(\Omega^R)$ is independent of $R\geq R'$, which we define as the \textit{$E$-cohomological Conley index $\ch_E(\Omega)$ of $\Omega$}. The aim of this section is to introduce a sufficient condition for an invariant set to be bounded.

\begin{defi}
Let $F:U\subset E\rightarrow E$ be an $\ls$-vector field. We say that $F$ satisfies the weak boundedness condition $\wbds$ on $\Omega \subset U$ if there exists $\epsilon > 0$ such that

\[F^{-1}(B(\epsilon)) \cap \Omega\]
is bounded. Moreover, $F:I\times U\rightarrow E$ is a $\wbd$-homotopy on $\Omega$ if $$\bigcup\limits_{s\in[0,1]}F^{-1}_s(B(\epsilon)) \cap \Omega$$ is bounded, where $F_s:=F(s,\cdot):U\subset E\rightarrow E$.\\
If, in addition, $F=\nabla f$ for a $C^1$-functional, or $F_\lambda=\nabla f_\lambda$ for a continuous map $f:I\times U\rightarrow\mathbb{R}$ of $C^1$ functionals, then we say that $F$ is a gradient $\wbd$- vector field, or a gradient $\wbd$-homotopy on $\Omega$, respectively.  
\end{defi}

\noindent Let us point out that we use the name \textit{weak boundedness condition} in order to distinguish it from the stronger and more commonly used boundedness condition \bd, which requires that the preimage of any bounded set is bounded. Actually, a map between finite dimensional spaces is \bds if and only if it is \textit{proper}, however, in an infinite dimensional setting \bds implies that the map is proper, but there are proper maps which are not \bd (see \cite[p.4-5]{Bauer}).

\begin{prop}\label{prop-bounded}
Let $\eta$ be an $\ls$-flow on $U$ that is generated by a gradient vector field which is $\wbd$ on a closed set $\Omega \subset U$. Then the invariant set $\inv(\Omega,\eta)$ is bounded. Moreover, if $F:I\times U\rightarrow E$ is a gradient $\wbd$-homotopy on $\Omega$ and $\eta_\lambda$, $\lambda\in I$, the corresponding $\ls$-flows, then the invariant sets $\inv(\Omega,\eta_\lambda)$ are uniformly bounded.
\end{prop}
\noindent
We have already used in Theorem \ref{continuation} that $\ch_E(\Omega,\eta_\lambda)$ is independent of $\lambda$ if $\Omega$ is an isolating neighbourhood of $\eta_\lambda$ for all $\lambda\in I$, which was shown in \cite[Thm, 2.12]{Invariance}. If we now combine this fact with the previous Proposition, we obtain the following important corollary. 

\begin{cor}\label{cor-bounded}
Let $F:I\times U\rightarrow E$ be a gradient $\wbd$-homotopy on $\Omega$ and $\eta_\lambda$, $\lambda\in I$, the corresponding $\ls$-flows. Suppose that $\Omega$ is an unbounded isolating neighbourhood for every $\eta_\lambda$. Then the $E$-cohomological Conley index $\ch_E(\Omega,\eta_\lambda)$ is well defined and independent of $\lambda$. 
\end{cor}

\noindent
We shall prove Proposition \ref{prop-bounded} below in Section \ref{section-proof}. Although the proof relies on the assumption that the flow is of gradient type, we believe that this assumptions can be weakened to pseudo-gradient flows as in \cite[Appendix A]{Rabinowitz}. However, this is not needed in the present paper.


\section{The Arnold Conjecture on the Torus $T^{2n}$}

Let $T^{2n}$ denote the standard Torus of dimension $2n$ and $\omega_0$ its standard symplectic structure. Let $H\in C^2(S^1 \times T^{2n},\mathbb{R})$ be a $1$-periodic Hamiltonian and $X_H$ the induced vecor field on $T^{2n}$ given by

\[dH(\cdot)=\omega(X_H,\cdot).\]
We consider the Hamiltonian equation

\begin{align}\label{Hamiltonian}
\dot{x}(t)=X_H(x(t)),
\end{align}
and the aim of this section is to prove the following deep theorem. 

\begin{theorem}[Strong Arnold conjecture on $T^{2n}$]\label{theorem-Arnold}
For every $C^2$-Hamiltonian on $T^{2n}$ there exist at least $2n+1$ contractible solutions of \eqref{Hamiltonian}.
\end{theorem}

\noindent The above theorem was first proved by Conley and Zehnder in \cite{ConleyZehnder} (cf. also \cite{Hofer}). Let us point out that our proof is considerably shorter. The main point of our argument is to use Theorem \ref{thm-estimate} as well as the continuation principle from Theorem \ref{continuation} to reduce the problem to the case of a trivial Hamiltonian.\\
Let us further emphasize that there is a serious reason to look for new proofs of Theorem \ref{theorem-Arnold}. For example, the strong Arnold conjecture is still open on $T^{2n} \times \mathbb{C}P^m$ where a similar analytical setting might be introduced. To the best of our knowledge, the previous methods only work to some extent in this case (see, however, \cite{Oh} for partial results), and therefore it is worthwhile to develop new approaches.

\subsection{The Analytical Setting}
Before proving Theorem \ref{theorem-Arnold}, let us first recall the analytical setting from \cite{Hofer}. We start with the case of $\mathbb{R}^{2n}$ and consider the space of smooth loops $C^\infty(S^1,\mathbb{R}^{2n})$ in $\mathbb{R}^{2n}$. If we set $e_k(t):= e^{tk2\pi J}$, $k\in\mathbb{Z}$, where $J$ is the standard symplectic matrix, then any $x \in C^\infty(S^1,\mathbb{R}^{2n})$ is represented by its Fourier-series
\begin{align}\label{Fourier}
x(t) = \sum_{k \in \mathbb{Z}} x_ke_k(t).
\end{align}
The Sobolev space $H^{\frac12}(S^1,\mathbb{R}^{2n})$ is the Hilbert space which is obtained as the completion of $C^\infty(S^1,\mathbb{R}^{2n})$ with respect to the scalar product
$$\ip{x}{y}_s = \ip{x_0}{y_0} +  2\pi\sum_{k \in \mathbb{Z}}\abs{k}\ip{x_k}{y_k}.$$
There is an orthogonal decomposition

\[H^{\frac12}(S^1,\mathbb{R}^{2n})=Z_0\oplus Z^-\oplus Z^+\]
into a $2n$-dimensional subspace $Z_0$ and closed infinite-dimensional subspaces $Z^+$ and $Z^-$ which correspond to $k = 0$, $k > 0$ and $k < 0$ in the Fourier-series expansion \eqref{Fourier}, respectively. In what follows, we denote by $P_0$, $P^+$ and $P^-$ the corresponding orthogonal projections.\\
Now let $H\in C^2(S^1 \times \mathbb{R}^{2n},\mathbb{R})$ be a Hamiltonian such that $|H(x)| \leq C \cdot |x|^2$ at infinity and such that the second spatial derivative $H''$ is globally bounded.  We define a functional $\Phi_H: C^\infty(S^1,\mathbb{R}^{2n}) \to \mathbb{R}$ by the formula

\begin{align}\label{phiH}
\Phi_H(x) = a(x) - b(x) := \frac12 \int_0^1 \ip{-J\dot{x}(t)}{x(t)} \,dt - \int_0^1 H(t,x(t))\, dt.
\end{align}
The importance of $\Phi_H$ comes from the fact that the critical points of $\Phi_H$ are periodic solutions of the Hamilton equation \eqref{Hamiltonian}.
It is easy to see that $\Phi_H$ extends to $H^{\frac12}(S^1,\mathbb{R}^{2n})$, and 

\begin{align}\label{L+K}
\nabla \Phi_H=L+K
\end{align}
where $L = \nabla a = P^+ - P^-$ is a selfadjoint Fredholm operator and 
$K = -\nabla b = -j^* \nabla H$ is a compact map because of the compactness of the adjoint $j^\ast:L^2\rightarrow H^{\frac12}$ of the inclusion.\\
On a general manifold, it is a delicate problem to define spaces $H^{\frac12}(S^1,M)$ as $H^{\frac12}(S^1,\mathbb{R}^{2n})$ contains non-continuous functions which consequently have no local meaning. However, for a torus one can overcome this problem by using the universal covering $\mathbb{R}^{2n} \to T^{2n} = \mathbb{R}^{2n}/\mathbb{Z}^{2n}$. Then smooth Hamiltonians on $T^{2n}$
are in one-to-one correspondence with $\mathbb{Z}^{2n}$-invariant smooth Hamiltonians on $\mathbb{R}^{2n}$, where $\mathbb{Z}^{2n}$ acts on $\mathbb{R}^{2n}$ by translations. By a slight abuse of notation, we will denote by $H$ both the Hamiltonian on the torus and the Hamiltonian lifted to $\mathbb{R}^{2n}$. Note that the lifted Hamiltonian on $\mathbb{R}^{2n}$ is $\mathbb{Z}^{2n}$-invariant and therefore its second spatial derivative is bounded and it obviously satisfies the growth condition mentioned above. 
Now the corresponding functional $\Phi_H$ in \eqref{phiH} is $\mathbb{Z}^{2n}$-invariant as well, and therefore it descends to a functional on the quotient space
$$ \mathcal{M} := Z_0/\mathbb{Z}^{2n} \times Z^+ \times Z^- = T^{2n} \times Z^+ \times Z^-.$$


\subsection{Proof of Theorem \ref{theorem-Arnold}}
We suppose as in the previous section that $H\in C^2(S^1\times T^{2n},\mathbb{R})$ is a given Hamiltonian. Let us note at first that $F=\nabla \Phi_H$ in \eqref{L+K} is an $\ls$-vector field, even though the operator $L$ is not invertible. Indeed, if we write $F=\hat{L}+\hat{K}:=(L+P_0)+(K-P_0)$, where $P_0$ is the orthogonal projection onto the finite dimensional kernel of $L$ as introduced above, then $F$ is the sum of an invertible selfadjoint operator and a compact map.\\
Since we want to apply the $E$-cohomological Conley index, we need to work with an open subset of a Hilbert space rather than a Hilbert manifold. For that reason, let us embed $\mathcal{M}$ into $\hat{E} = \mathbb{R}^{4n} \times Z^+ \times Z^-$ in such a way that every $S^1$ in $T^{2n} = S^1 \times \ldots \times S^1$ is mapped to the unit circle in $\mathbb{R}^{2}$. We consider the open set $$U := D^{2n}_0\times Z^+ \times Z^- \subset\hat{E}$$ of $\hat{E}$, where $D_0=\{(x,y)\in\mathbb{R}^2:\, 0<x^2+y^2<4\}$ is a punctured disc of radius $2$ in $\mathbb{R}^2$, and we let $\pi: \mathcal{N}\to\mathcal{M}$ be the standard projection to $T^{2n}$ on $D^{2n}_0$ and the identity on $Z^+$ and $Z^-$. The map $\Phi_H$ can be extended to $U$ by
$$\Psi_H(x) =\Phi_H(\pi(x)) + \sum_{i=1}^{2n} (1-r_i(x))^2,$$
where $r_i(x)$ denotes the polar coordinate in $\mathbb{R}^2$ of the projection of $x\in U$ to the $i$-th component of $(\mathbb{R}^2)^{2n}$. Note that the extension is done in such a way that $\Psi_H$ and $\Phi_H$ have the same critical points. We denote by $\tilde{K}$ the compact operator which is the sum of $\hat{K}$ and $\nabla(\sum_{i=1}^{2n}(1-r_i(x))^2)$.\\
Now Theorem \ref{theorem-Arnold} can be obtained as follows. We note first that $\nabla\Psi_H=\hat{L}+\tilde{K}$ is an $\ls$-vector field, and the negative and positive spectral subspaces of the selfadjoint isomorphism $\hat{L}$ are given by 

\[E^+= \mathbb{R}^4\oplus Z^+,\quad E^-=Z^-.\]
We consider the family $\{H_\lambda = (1-\lambda)H: \lambda \in [0,1]\}$ of Hamiltonians, and obtain a corresponding family of functionals $\Psi_\lambda:U\rightarrow\mathbb{R}$. Let us denote by $\hat{F}_\lambda:=\nabla\Psi_\lambda:U\subset\hat{E}\rightarrow\hat{E}$ the corresponding family of $\ls$-vector fields defined on the unbounded open subset $U$ of $\hat{E}$.

\begin{lemma}\label{lemma:arnold_bd} 
For every bounded set $B\subset\hat{E}$ the set $\bigcup_{\lambda \in [0,1]} \hat{F}^{-1}_{\lambda}(B)\subset U$ is bounded. In particular, $\hat{F}$ is a $\wbd$-homotopy on every $\Omega \subset U$.
\end{lemma}
\begin{proof}
Suppose on the contrary that there exists a sequence $\{(\lambda_n,x_n)\}\subset I\times U$ and a constant $c >0$ such that $\norm{x_n} \to \infty$ and  $\norm{\hat{F}_{\lambda_n}(x_n)} < c$. As $P^+x_n$ and $P^-x_n$ are orthogonal for all $n$, we have 
$$c > \norm{\hat{F}_{\lambda_n}(x_n)} \geq \frac{1}{2}\norm{P^+x_n} + \frac{1}{2}\norm{P^-x_n} - \norm{\tilde{K}_{\lambda_n}(x_n)}.$$
 Since the family $\{H_{\lambda_n}\}$ is uniformly bounded and $0 < r_i < 2$ we see that the norm $ \norm{\tilde{K}_{\lambda_n}(x_n)}$ is bounded. On the other hand, if $\norm{x_n} \to \infty$, then $\norm{P^+x_n} \to \infty$ or $\norm{P^-x_n} \to \infty$, which is a contradiction.
\end{proof}
\noindent
Let now $\eta_\lambda$ be the flow on $U$ generated by $\hat{F}_\lambda$. Clearly, 
\[\Omega=A^{2n}\times Z^+ \times Z^-,\]
where 
\[A:=\left\{(x,y)\in\mathbb{R}^2:\, \frac{1}{2}\leq\sqrt{x^2+y^2}\leq\frac{3}{2}\right\},\]
is an unbounded isolating neighbourhood for $\eta_\lambda$ for every $\lambda$ in the sense of Definition \ref{isolnbhdunb}. By Corollary \ref{cor-bounded} and Lemma \ref{lemma:arnold_bd}, the $E$-cohomological Conley index $\ch(\Omega,\eta_\lambda)$ is well defined and independent of $\lambda$. However, for $\lambda = 1$ the flow corresponds to the trivial Hamiltonian $H\equiv 0$, and consequently we now want to compute the cup-length for $\eta_1$. We note at first that the pair

\[(X,A):=(A^{2n}\times B(Z^+)\times B(Z^-), A^{2n}\times B(Z^+)\times\partial B(Z^-))\]
is an index pair, and if $V\subset Z^-$ is of finite dimension, then

\[(X_V,A_V)=(A^{2n}\times B(Z^+)\times B(V), A^{2n}\times B(Z^+)\times\partial B(V)),\]
where $B(V)$ denotes the unit ball in $V$. Hence we get for $k\in\mathbb{Z}$

\[H^k(X_V,A_V)=H^k(X_V/A_V)=H^k(S(V)\wedge T^{2n})=H^{k-\dim(V)}(T^{2n}),\]
where $S(V)$ denotes the unit sphere in $V$. Moreover, if $W\supset V$ is another finite dimensional subspace, then the Mayer-Vietoris homomorphism $\Delta^k_{V,W}$ mapping

\[H^{k+\dim(V)}(X_V,A_V)=H^{k+\dim(V)}(S(V)\wedge T^{2n})\]
to 
\[H^{k+\dim(W)}(X_W,A_W)=H^{k+\dim(W)}(S(W)\wedge T^{2n})\]
is by definition just the suspension isomorphism. Hence we obtain

\[H^\ast_E(X,A)=H^\ast(T^{2n}).\]
Finally, to find the cup-length, we note at first that for our isolating neighbourhood $\Omega$, $R'\geq 2$, and any finite dimensional subspace $V\subset Z^-$

\[H^\ast(\Omega_V^{R'})=H^k(A^{2n}\times B(Z^+;R')\times B(V;R'))=H^\ast(T^{2n}),\]
where $B(Z^+;R')$ and $B(V;R')$ denote the balls of radius $R'$. Hence $\CL(\Omega^{R'},\eta_1)$ is just the ordinary cup-length of the torus $T^{2n}$, which is $2n+1$. Now let $R>0$ be so large that $\Omega^R$ is an isolating neighbourhood of $\eta_\lambda$ for all $\lambda\in I$. We obtain from Theorem \ref{continuation} that $\CL(\Omega^R,\eta_0)=2n+1$. By Theorem \ref{thm-estimate}, this is a lower bound for the number of critical points of $\Phi_H$ in $\Omega^R$, and so we have proved the Arnold conjecture on $T^{2n}$.



\section{Proof of Theorem \ref{thm-estimate} and Proposition \ref{prop-bounded}}\label{section-proof}
In this section we prove Theorem \ref{thm-estimate} and Proposition \ref{prop-bounded}. In what follows, we will use that if $F:U\subset E\rightarrow E$ is an $\ls$-vector field such that $F=\nabla f$ for a functional $f:U\rightarrow\mathbb{R}$, and $\{\eta^t(x):\, t\in\mathbb{R}\}$ is a trajectory such that $\eta^t(x)\in\Omega$ for all $t\in\mathbb{R}$ and some bounded set $\Omega\subset
U$, then the limits $\alpha(x)$ and $\omega(x)$ are contained in the set of critical points of $f$. This is, for example, a simple consequence of the property (C) of gradient $\ls$-vector fields that was discussed in \cite{Invariance}.

\subsection{Proof of Theorem \ref{thm-estimate}}
We will need the following two properties of the cup-length $\CL$ that we introduced in Definition \ref{cuplength}. As the proofs are purely algebraic, we leave it to the reader to check that they follow by obvious modifications from \cite[Lemma 2.2 \& 2.3]{Uss}.

\begin{lemma}\label{lem:cupLengthProperty}
If $B \subset A \subset X \subset Y$ are closed and bounded subsets of $E$, then
$$\CL(Y;X,B) \leq \CL(Y;X,A) + \CL(Y;A,B).$$
\end{lemma}

\begin{lemma}\label{lem:cupLengthProperty2}
If $A \subset X \subset Y_1 \subset Y_2$ are closed and bounded subsets of $E$, then
$$\CL(Y_2;X,A) \leq \CL(Y_1;X,A).$$
\end{lemma}
\noindent
Now let us consider an isolating neighbourhood $\Omega$ for the flow $\eta$ generated by the gradient of the function $f:U \rightarrow\mathbb{R}$ in Theorem \ref{thm-estimate}. As we suppose that there are only finitely many critical points of $f$ in $\Omega$, the set of critical values $\crit(f,\Omega)$ is finite as well, say, $c_1<\ldots<c_k$. Let $M_i\subset\Omega$ denote the set of stationary points with values $c_i$, and set for $1\leq i\leq j\leq k$

\[M_{ij}=\{x\in \Omega:\omega(x)\cup\alpha(x)\subset M_i\cup M_{i+1}\cup\ldots\cup M_j\},\]
where $\alpha(x)$ and $\omega(x)$ denote as above the $\alpha$ and $\omega$ limits of $x\in E$ under the flow $\eta$. Note that $M_{1k}$ consists of all the critical points of $f$ inside $\Omega$ and all the orbits connecting them. Consequently, $$M_{1k} = \inv(\Omega,\eta).$$ Now let $(X,A)$ be an index pair for $M_{1k}$.

\begin{lemma}[Morse filtration]
There exist sets 

\[X_0=A\subset X_1\subset\ldots\subset X_k=X\]
such that $(X_j,X_{i-1})$ is an index pair for $M_{ij}$.
\end{lemma}

\begin{proof}
We let $b_i \in (c_i,c_{i+1})$, $i=1,\ldots,k-1$ be regular values of $f$, set $b_k=\infty$, and define $X_0=A$ as well as $X_i:= X \cap f^{-1}(-\infty,b_i]$, $i=1,\ldots,k$. Then it is readily seen that $(X_j,X_{i-1})$ is an index pair for $M_{ij}$ as $M_{ij}$ consists of all critical points $x$ such that $f(x)\in\{c_i,\ldots, c_j\}$ and all the orbits connecting them.
\end{proof}
\noindent
If we now apply Lemma \ref{lem:cupLengthProperty} $k$ times, we get

\begin{align}\label{clinequ}
\CL(\Omega;X,A) \leq \sum\limits_{i=1}^k \CL(\Omega;X_i,X_{i-1}). 
\end{align}
On the other hand, $(X_i,X_{i-1})$ is an index pair for $M_{ii}$, which is a set consisting of a finite number of stationary points. Therefore we can choose an isolating neighbourhood $\Omega_i$ for $M_{ii}$, where $\Omega_i$ is a disjoint union of discs. If now $(X_i',X_{i-1}')$ is an index pair for $M_{ii}$ such that $X_i' \subset  \Omega_i$, then by Corollary \ref{cor-moduleiso} and Lemma \ref{lem:cupLengthProperty2} 
\[\CL(\Omega;X_i,X_{i-1}) = \CL(\Omega;X_i',X_{i-1}') \leq \CL(\Omega_i;X'_i,X'_{i-1}) \leq 1\]
where the last inequality follows from the fact that the groups $H_0^{q>0}(\Omega_i)$ are trivial. Hence, by \eqref{clinequ},

\[\CL(\Omega;X,A)\leq k\]
and Theorem \ref{thm-estimate} is shown, as $k$ is the number of critical values of $f$ in $\Omega$.


\subsection{Proof of Proposition \ref{prop-bounded}}
We note at first the following simple lemma.

\begin{lemma}\label{lem:compactZeosFamily}
Let $\Lambda$ be a compact space of parameters and let $\{F_\lambda\ = L + K_\lambda\}_{\lambda \in \Lambda}$ be a continuous family of $\ls$-vector fields $F_\lambda: U\subset E\rightarrow E$. 
If $\bigcup\limits_{\lambda \in \Lambda} F^{-1}_\lambda(0)\subset U$ is bounded, then it is compact.
\end{lemma}

\begin{proof}
Let $\{x_i\}_{i\in\mathbb{N}}$ be a sequence in $\bigcup\limits_{\lambda \in \Lambda} F^{-1}_\lambda(0)$. As $L$ is Fredholm, there exists a linear bounded operator $T:E\rightarrow E$ and a linear compact operator $C:E\rightarrow E$ such that $TL=I_E+C$. Hence

\[0=TF(x_i)=x_i+Cx_i+TK_{\lambda_i}(x_i),\quad i\in\mathbb{N},\]
and so $x_i=-(C+TK_{\lambda_i})(x_i)$. As $C+TK_{\lambda_i}$ and $\Lambda$ are compact, we see that $\{x_i\}_{i\in\mathbb{N}}$ contains a convergent subsequence. Consequently, we have shown that $\bigcup\limits_{\lambda \in \Lambda} F^{-1}_\lambda(0)$ is relatively compact, and as it is closed, we see that it is compact. 
\end{proof}
\noindent
Our proof of Proposition \ref{prop-bounded} now follows standard arguments in Morse theory. The reader may compare our argument with the compactness proof in \cite[pp. 56-57]{Schwarz}.\\
By assumption \wbd, there exists $\epsilon > 0$ such that $Y:=F^{-1}(B(\epsilon))\cap\Omega$ is bounded. Suppose that $r_0>0$ is such that $Y\subset B(r_0)$. Now let $X:=\inv(\Omega,\eta)\subset U$ be the invariant set for the negative gradient flow of $f$ which is given by all critical points of $f$ and the flow lines between them. For proving Proposition \ref{prop-bounded}, we need to show that it is bounded. As this is certainly true if $X\subset B(r_0)$, we can assume that $X\setminus B(r_0)\neq\emptyset$ and consider a point $x\in X$ such that $x\not\in B(r_0)$. Let now $u:\mathbb{R}\rightarrow E$ be the trajectory of the flow starting at $x$, and let $y \in \omega(x) \subset \crit(f)$. Note that $y\in Y\subset B(r_0)$ as $F(y)=0$.\\
Let $t_0\in(-\infty,0)$ such that $u(t_0)\in\partial B(r_0)$ but $u(t)\not\in\partial B(r_0)$ for all $t\in(t_0,0)$. Then

\[\|x-y\|\leq \|y-u(t_0)\|+\|u(t_0)-x\|\leq 2r_0+\int^0_{t_0}{|\dot{u}(s)|ds}\]
and we see that we need to find a bound on $\int^0_{t_0}{|\dot{u}(s)|ds}$ which is independent of $u$. If we set $l(s)=\int^s_{t_0}{|\dot{u}(s)|ds}$, then

\begin{align*}
\frac{dl}{ds}(s)=|\dot{u}(s)| &= |\nabla f(u(s))|,\\
\frac{d(f(u(s)))}{ds}(s)& = - |\nabla f(u(s))|^2
\end{align*}
and so

\[\frac{dl}{ds}(s) \leq - \frac{1}{\epsilon}\frac{d(f(u(s)))}{ds}(s)\]
for every $s\in(t_0,0]$, where we use that $u(s)\notin B(r_0)\supset Y$ and so $|\nabla f(u(s))|\geq\varepsilon$ for these values of $s$. Therefore

\begin{align*}
\|x-y\|& \leq  2r_0 + \int_{t_0}^0 |\dot{u}(s)| \, ds = 2r_0 + \int_{t_0}^0 \frac{dl}{ds}(s) \, ds \\
    &\leq 2r_0 - \frac{1}{\epsilon}\int_{t_0}^0  \frac{d(f(u(s)))}{ds}(s) \, ds =  2 r_0 + \frac{1}{\epsilon}[f(u(t_0)) - f(u(0))].
\end{align*}
As $f$ decreases along flow lines, 

\[2r_0 +\frac{1}{\varepsilon} [f(u(t_0)) - f(u(0))]\leq 2r_0+\frac{1}{\epsilon}[f(z)-f(y)],\]
where $z\in\alpha(x)$. As the set of critical points of $f$ in $\Omega$ is a subset of the bounded set $Y$, it is compact by Lemma \ref{lem:compactZeosFamily}, and so there exists $r_1>0$ such that

\[r_1 \geq \frac{1}{\epsilon}|f(x_1) - f(x_2)|\]
for all critical points $x_1$ and $x_2$ of $f$ in $\Omega$. Hence $X\subset B(0,R)$ for $R:=2r_0+r_1$ which shows that $X$ is bounded as $r_0$ and $r_1$ were chosen independently of $x\in X$.\\
To see that the second part of the Proposition is true, note that if we have a continuous family parametrised by a compact space $\Lambda$, then by the definition of \wbd-homotopy and Lemma \ref{lem:compactZeosFamily}, $r_0$ and $r_1$ can be chosen independently of $\lambda\in\Lambda$.





\begin{minipage}{1.0\textwidth}
\begin{minipage}{0.4\textwidth}
Maciej Starostka\\
Polish Academy of Sciences\\
and\\
Gdansk University of Technology\\
Poland\\
E-mail: maciejstarostka@impan.pl
\end{minipage}
\hfill
\begin{minipage}{0.4\textwidth}
Nils Waterstraat\\
School of Mathematics,\\
Statistics \& Actuarial Science\\
University of Kent\\
Sibson Building\\
Parkwood Road\\
Canterbury\\
Kent CT2 7FS\\
UNITED KINGDOM\\
E-mail: n.waterstraat@kent.ac.uk
\end{minipage}
\end{minipage}

%

\end{document}